\documentclass[11pt]{amsart}
\usepackage{amsmath,amsfonts,amsthm,amssymb,amsthm,graphicx,mathtools,amscd,enumerate,paralist,bm,hyperref}

\usepackage[mathscr]{eucal}
\usepackage{graphics,color,epsfig}
\usepackage[nocompress]{cite}

\DeclareSymbolFont{iwonaletters}{OML}{iwona}{m}{it}

\theoremstyle{plain}
\newtheorem{theorem}{Theorem}[section]
\newtheorem{lemma}[theorem]{Lemma}

\newtheorem{definition}[theorem]{Definition}

\newtheorem{proposition}[theorem]{Proposition}

\newtheorem{assumption}[theorem]{Assumption}

\newcommand{\la}{\lambda}
\newcommand{\eps}{\varepsilon}

\newcommand{\al}{\alpha}

\newcommand{\kap}{\kappa}
\newcommand{\sig}{\sigma}
\newcommand{\del}{\delta}

\newcommand{\Del}{\mathnormal{\Delta}}

\newcommand{\La}{\mathnormal{\Lambda}}

\newcommand{\Om}{\mathnormal{\Omega}}

\newcommand{\N}{{\mathbb N}}

\newcommand{\R}{{\mathbb R}}

\newcommand{\E}{{\mathbb E}}

\newcommand{\PP}{{\mathbb P}}

\newcommand{\calA}{{\mathcal A}}

\newcommand{\calC}{{\mathcal C}}

\newcommand{\calF}{{\mathcal F}}

\newcommand{\calP}{{\mathcal P}}
\newcommand{\calR}{{\mathcal R}}
\newcommand{\calS}{{\mathcal S}}

\newcommand{\calW}{{\mathcal W}}

\newcommand{\bX}{\mathbf{X}}
\newcommand{\bY}{\mathbf{Y}}

\newcommand{\frp}{\mathfrak{p}}

\renewcommand{\proof}{\noindent{\bf Proof.\ }}

\newcommand{\lan}{\langle}
\newcommand{\ran}{\rangle}

\newcommand{\oo}{\overline}

\newcommand{\w}{\wedge}

\newcommand{\To}{\Rightarrow}

\newcommand{\iy}{\infty}

\newcommand{\noi}{\noindent}

\newcommand{\sym}{{\rm sym}}

\begin{document}

\title[]{
High-dimensional limits for reflected Brownian motion in the orthant
}

\author{Rami Atar}
\address{Viterbi Faculty of Electrical \& Computer Engineering
\\
Technion--Israel Institute of Technology
} 
\email{rami@technion.ac.il}


\date{\today}

\begin{abstract}
We study interacting Brownian particles on the half-line whose interaction occurs through boundary local times at the origin. The particle system is given by
\[
X_i^n(t)=X^n_{0,i}+W_i^n(t)+L_i^n(t)
+\frac{1}{n-1}\sum_{j\ne i}\rho^n_{ij}L_j^n(t),
\qquad i\in[n],\ t\ge0,
\]
where the initial conditions are exchangeable, the driving Brownian motions $W_i^n$ are i.i.d., and $L_i^n$ denotes the boundary local time of $X_i^n$ at zero. For each fixed coefficient array $\{\rho^n_{ij}\}$, the system can be viewed as a semimartingale reflected Brownian motion in the orthant. We first consider the homogeneous case $\rho^n_{ij}=a$. In this case, global well-posedness holds under the completely-$\mathcal S$ condition $a>-1$. We prove propagation of chaos under this condition; the subregime $a\in(-1,0]$, in the homogeneous setting, was previously covered as part of the results of \cite{baker2025particle}.
The limiting process is the nonlinear reflected Brownian motion
\[
\oo X(t)=\oo X_0+\oo W(t)+\oo L(t)+a\mathbb E[\oo L(t)],
\qquad t\ge0.
\]
We also treat heterogeneous random coefficients $\rho^n_{ij}$, assumed to have mean $a$, support in a compact subset of $(-1,1)$, and to be independent across $j$ for each $i$.
In both the quenched and annealed settings, the particle system converges to the same McKean--Vlasov limit as in the homogeneous case. The model is motivated by large Jackson networks in heavy traffic.
\end{abstract}

\maketitle

\section{Introduction}

Consider an $n$-dimensional semimartingale reflected Brownian motion (SRBM) in the orthant $[0,\iy)^n$, $n\ge2$, described by the stochastic differential equation (SDE)
\begin{equation}\label{01a}
X_i^n(t)=X^n_{0,i}+W_i^n(t)+L_i^n(t)+\frac{1}{n-1}\sum_{j\ne i}aL_j^n(t),
\qquad i\in[n], \ t\ge0.
\end{equation}
Here, $a\in\R$ is a parameter, $W^n=(W^n_i)_{i\in[n]}$ is an $n$-dimensional Brownian motion (BM) the components of which are i.i.d.\ BM with drift and diffusion coefficients $b\in\R$ and $\sig>0$, and $L^n=(L^n_i)_{i\in[n]}$ is the boundary process: a componentwise nondecreasing process satisfying $\int_{[0,\iy)}X_i(s)dL_i(s)=0$ a.s.
A solution to the SDE exists globally in time if and only if $a>-1$, in which case it is unique in law; see Section \ref{sec2} for precise definitions and further details. Let
\[
\mu^n_t=\frac{1}{n}\sum_{i\in[n]}\del_{X^n_i(t)}.
\]
Assume exchangeability of $\{X^n_{0,i}\}_{i\in[n]}$ for each $n$, and convergence $\mu^n_0\to\mu_0$ in probability to a deterministic limit. It is then natural to regard \eqref{01a} as a system of Brownian particles on the half-line, interacting through their local time at the origin. The associated McKean--Vlasov process, which we call the {\it nonlinear RBM}, is the unique solution to
\begin{equation}\label{MVSDE}\tag{NLR}
\begin{split}
\oo X(t) &= \oo X_0+\oo W(t)+\oo L(t)+a\,\la(t)
\\
\la(t) &= \E \oo L(t),
\end{split}
\end{equation}
with $\oo X_0\sim \mu_0$, $\oo W$ a  BM with drift and diffusion coefficients $b$ and $\sig$, and $\oo L$ a boundary term; see Section \ref{sec2} for details.
Our first result proves propagation of chaos: for each fixed $k\in\N$, the tuple $\{X^n_i\}_{i\in[k]}$ converges in distribution to a tuple of $k$ independent copies of the nonlinear RBM, and the empirical distribution $\mu^n_\cdot$ converges in probability to $\mu_\cdot$, where $\mu_t$ is the law of $\oo X(t)$.
The subregime $a\in(-1,0]$ was previously covered as part of the results of \cite{baker2025particle}.

We next consider a heterogeneous version with random reflection coefficients, where the SRBM is given by
\begin{equation}\label{01b}
X_i^n(t)=X^n_{0,i}+W_i^n(t)+L_i^n(t)+\frac{1}{n-1}\sum_{j\ne i}\rho^n_{ij}L_j^n(t),
\qquad i\in[n], \ t\ge0.
\end{equation}
Here, the coefficients $\rho^n_{ij}$ are assumed to be supported in a compact subset of $(-1,1)$, have mean $a$, and for every $n\ge2$ and $i\in[n]$, the $n-1$ random variables $\rho^n_{ij}$, $j\in[n]-\{i\}$ are mutually independent.
Our second result shows that both the quenched and annealed limits are governed by the same nonlinear RBM \eqref{MVSDE}. In particular, at the McKean--Vlasov scale, the random environment enters only through the mean coefficient $a$.

\subsection*{Motivation}

A concrete source of motivation comes from open Jackson queueing networks and their generalizations. Consider a network with $n$ single-server stations. After completing service at station $i$, a job is routed to station $j$ with probability $p_{ij}$, or leaves the system with probability $1-\sum_j p_{ij}$. A classical heavy-traffic theorem of Reiman \cite{reiman1984open} shows that, for fixed $n$, the diffusion-scaled queue-length process converges to a semimartingale reflected Brownian motion in $\mathbb R_+^n$ with reflection matrix $(R_{ij})=(\delta_{ij}-p_{ji})$. This places the particle systems studied here in the same family of high-dimensional reflected diffusions. With the mean-field normalization in \eqref{01b}, the Jackson-network routing probabilities and the reflection coefficients are related via $p^n_{ji}=-\rho^n_{ij}/(n-1)$ for $i\ne j$.

In the special case where $-1<-1+\eps_\rho\le\rho^n_{ij}\le0$ and $p^n_{ij}$ of the above form, the McKean--Vlasov limit obtained in this paper describes the asymptotics of a Jackson network in which the heavy-traffic diffusion limit is taken first, with $n$ fixed, and then $n$ is sent to infinity. This motivates the further question of whether the same limit is obtained when the heavy-traffic and many-station limits are taken jointly, or in the reverse order. At the queueing-network level such interchange problems can be delicate, particularly for generalized networks whose queue-length processes are not Markovian.

In addition to the queueing network motivation, the model fits within the literature on mean-field limits and propagation of chaos for particle systems with local-time interactions; several related works are discussed below.

\subsection*{Related work}

Since the pioneering work of McKean \cite{mckean1966class}, McKean--Vlasov limits and propagation of chaos have become central topics in probability theory. For classical accounts, we refer to \cite{gar88} and \cite{sznitman1991topics}.

McKean--Vlasov limits for reflected diffusions were studied in \cite{sznitman1984nonlinear}, where the particle system is given by
\[
dX^n_i(t)=b(X^n_i(t),\mu^n_t)dt+dW_i(t)+dL^n_i(t),
\qquad i\in[n],
\]
with each $X^n_i$ constrained to lie in a closed subset of $\R^d$, and $L^n_i$ denoting the boundary term enforcing normal reflection. Numerous extensions of this framework have since been developed; for recent contributions, see e.g.\ \cite{hinds2025well, wei2026mckean}. In these works, the interaction acts through the effect of the empirical law on the drift and diffusion coefficients, but not through the local-time or boundary terms themselves.

To our knowledge, McKean--Vlasov limits of particle systems interacting through local time first appeared in \cite[Chapter II]{sznitman1991topics}, where particles on the real line satisfy
\[
dX^n_i(t)= dW_i(t)+\frac{c}{n}\sum_{j\ne i}dL^0(X^n_i-X^n_j)(t),
\qquad i\in[n],
\]
with $L^0$ denoting the symmetric local time at $0$.
Related models have also been studied in the context of particle systems subject to constraints on the empirical law. In \cite{coghi2022mckean}, a class of such systems is considered, including the following representative example: particles reflected in the unit interval satisfy
\[
dX^n_i(t)=dW_i(t)+dK^n(t)+dL^n_i(t),
\qquad i\in[n],
\]
together with the constraint $n^{-1}\sum_iX^n_i(t)=q\in(0,1)$. Here, $L^n_i$ are boundary terms for reflection in $[0,1]$, and $K^n$ is a control term enforcing the constraint. In particular, one has $dK^n=n^{-1}\sum_i dW_i+n^{-1}\sum_i dL^n_i$,
so that, similarly to \eqref{01a}, the interaction term is given by the empirical average of boundary terms. We refer to \cite{coghi2022mckean} for further references on particle systems with macroscopic constraints leading to interactions through local time.
See also \cite{barnes2020hydrodynamic}, which studies a particle system with reflection in which the interaction occurs through a term of the form $n^{-1}\sum_iL^n_i(t)dt$; there, the average boundary local time enters the dynamics through the drift.

Most closely related to the present paper is \cite{baker2025particle}. For fixed $n$, the work studies \eqref{01b} with deterministic, non-positive coefficients $\rho^n_{ij}$, including regimes in which global solutions need not exist. In such regimes, solutions exist only up to a breakdown time, which is characterized there. In addition, in the homogeneous case $\rho^n_{ij}=a$, $a\le0$, \cite{baker2025particle} studies the McKean--Vlasov limit. The subcase $a\in(-1,0]$ overlaps with the results of the present paper, while for $a<-1$ the limiting dynamics break down in finite time (and a more complicated behavior occurs when $a=-1$).

Finally, there is related work on mean-field limits for queueing models. These limits typically correspond to sending the number of servers, stations, or nodes to infinity at the level of the queueing system itself. Examples include strategic servers in a mean-field game setting \cite{bayraktar2019rate} and a randomized load-balancing algorithm studied in \cite{ata-wol}. Generalized Jackson networks have been studied under mean-field asymptotics in \cite{rybko2016stationary}, as part of a line of research on the Poisson hypothesis: as the number of servers increases, the inflow to each node becomes asymptotically Poisson. See also the earlier works \cite{rybko2005poisson,rybko2005poisson2,baccelli1992mean} and the references therein. This direction is distinct from the heavy-traffic motivation described above: it studies many-node limits of the original queueing networks without diffusion scaling of the queue-length or delay processes, and therefore does not lead to heavy-traffic reflected Brownian approximations.

\subsection*{Notation}

Denote $\R_+=[0,\iy)$. For $x\in\R$, let $x^{\pm}=\max(\pm x,0)$. For $n\in\N$, write $\mathbf{1}_n$ for the column vector in dimension $n$ all entries of which are $1$, and $\mathbf{I}_n$ for the $n\times n$ identity matrix.
Equip $\R^n$ with the Euclidean norm $|\cdot|$.

For a Polish space $(E,d)$, let $C(\R_+,E)$ be the space of continuous paths $\R_+\to E$ equipped with the topology of convergence uniformly on compacts.
Let $\calP(E)$ denote the set of probability measures on $E$, equipped with the topology of weak convergence.
For $\mu\in\calP(E)$, let $\calP_\sym(E^n;\mu)$ be the set of symmetric (i.e., permutation-invariant) probability measures on $E^n$ with marginal $\mu$. Denote by $\To$ convergence in distribution.

Let $C_+(\R_+,\R_+)$ denote the set of functions in $C(\R_+,\R_+)$ that are nondecreasing and start at 0. Let $C^{(n)}=C(\R_+,\R^n)$ and $\calC^{(n)}$ the Borel $\sig$-algebra on $C^{(n)}$.
Let $S^{(n)}$ denote the $n$-dimensional orthant $\R_+^n$ and $\calS^{(n)}$ the Borel $\sig$-algebra on $S^{(n)}$. Let $C_{S_n}(\R_+,\R^n)=\{f\in C(\R_+,\R^n):f(0)\in S_n\}$.
For $f:\R_+\to \R^n$, $0\le\del\le t$, let
\begin{align*}
\|f\|_t &=\sup\{|f(s)|:s\in[0,t]\}
\\
w_t(f,\del) &=\sup\{|f(s)-f(u)|:s,u\in[0,t],\,|s-u|\le\del\}.
\end{align*}
Finally, for $v\in\R^n$, $n\ge2$,
\[
\lan v\ran:=\frac{1}{n}\sum_{i\in[n]}v_i,
\qquad
\lan v\ran_i:=\frac{1}{n-1}\sum_{j\ne i}v_j.
\]
With this notation, \eqref{01a}, for example, can be written as
\begin{equation}\label{51}
X^n_i(t)=X^n_{0,i}+W^n_i(t)+L^n_i(t)+a\lan L^n(t)\ran_i.
\end{equation}

\section{Setting and results}\label{sec2}

For $b\in\R$ and $\sigma>0$, an $n$-dimensional BM with drift $b\mathbf{1}_n$ and diffusion matrix $\sig\mathbf{I}_n$, starting at $0$, will be called an $(n,b,\sig)$-BM, and a $(1,b,\sig)$-BM will be called a $(b,\sig)$-BM. We will consider here only SRBM driven by $(n,b,\sig)$-BM. Following \cite{rei-wil,tay-wil,wil98inv}, such an SRBM is defined as follows.

\begin{definition}[SRBM]
\label{def1}
Given $n\in\N$, $b\in\R$, $\sig>0$, an $n\times n$ matrix $R$ and a probability measure $\nu$ on $(S^{(n)},\calS^{(n)})$, an SRBM associated with the data $(n,b,\sig,R,\nu)$ is an $\{\calF_t\}$-adapted, $n$-dimensional process $X$ defined on some filtered probability space $(\Om,\calF,\{\calF_t\},\PP)$ such that $\PP$-a.s.,
\[
\text{$t\mapsto X(t)$ lies in $C(\R_+,S^{(n)})$ and }
X=X_0+W+RL,
\]
and under $\PP$,
\begin{itemize}
\item
$X_0$ has distribution $\nu$,
\item
$W$ is an $(n,b,\sig)$-BM and, for each $i\in[n]$, $\{W_i(t)-bt,t\ge0\}$ is an $\{\calF_t\}$-martingale,
\item
$L$ is an $n$-dimensional $\{\calF_t\}$-adapted process such that $\PP$-a.s.,
\begin{equation}\label{02}
\text{$t\mapsto L_i(t)$ lies in $C_+(\R_+,\R_+)$ and } \int_0^\iy X_i^n(t)dL_i^n(t)=0, \qquad i\in[n].
\end{equation}
\end{itemize}
\end{definition}
We will refer to $R$ and the {\it reflection matrix}, and to $L$ as the {\it boundary process}.

The existence and uniqueness in law of SRBM in the orthant was studied in \cite{rei-wil, tay-wil}, where a necessary and sufficient condition was given in terms of the following key property of the reflection matrix:

{\it An $n\times n$ matrix $R$ is said to be completely-$\calS$ if for every nonempty $J\subset[n]$ there exists a vector $(x_i)_{i\in J}\in\R_+^J$ such that}
\[
\sum_{j\in J}R_{ij}x_j>0 \ \text{for all } i\in J.
\]

The following was shown in \cite[Corollary 1.4]{tay-wil} and \cite[Theorem 3.1]{wil98inv}.
\begin{theorem}[\cite{tay-wil,wil98inv}]
\label{th-w}
Let $n\in\N$, $b\in\R$, $\sig>0$, $R$ an $n\times n$ matrix, and $\nu$ a probability measure on $(S^{(n)},\calS^{(n)})$. Then there exists an SRBM associated with $(n, b,\sig,R,\nu)$ if and only if $R$ is a completely-$\calS$ matrix, in which case uniqueness in law holds for the pair $(X,L)$, where $X$ is an SRBM and $L$ the corresponding boundary process.
\end{theorem}

\subsection{Deterministic reflection}

We shall focus on \eqref{51}, in which $R$ is the $n\times n$ matrix
\begin{equation}\label{05}
R^{(n)}_{ij}=\begin{cases}
1& i=j\\
\frac{a}{n-1} &i\ne j.
\end{cases}
\end{equation}

\begin{lemma}\label{lem01}
Fix $n\ge2$. Then $R^{(n)}$ defined by \eqref{05} is completely-$\calS$ if and only if $a>-1$.
\end{lemma}

\proof
Suppose $a>-1$. Let $J\subset[n]$. Then for the vector $x\in\R_+^J$ all the entries of which are $1$, one has for all $i\in J$,
\[
x_i+\sum_{j\in J,j\ne i}R^{(n)}_{ij}x_j > 1 -\sum_{j\in J,j\ne i}\frac{1}{n-1} = 1-\frac{|J|-1}{n-1}\ge0.
\]
Next, suppose $a\le -1$. If for some $x\in\R_+^n$ one has $x_i+\sum_{j\ne i}R^{(n)}_{ij}x_j>0$ for all $i$ then
\(
x_i-(n-1)^{-1}\sum_{j\ne i}x_j>0.
\)
Summing over $i$ gives a contradiction.
\qed

We now fix $a>-1$, $b\in\R$, $\sig>0$, and a probability measure $\mu_0$ on $(\R_+,\calS^{(1)})$, and for each $n$, a probability measure
$\nu^{(n)}\in\calP_\sym(S^{(n)};\mu_0)$. For each $n$ we let $(X^n,L^n,W^n,X^n_0)$ be an SRBM and the corresponding boundary process, BM and initial condition for the data $(n,b,\sig,R^{(n)},\nu^{(n)})$. We assume without loss of generality that they are defined on a common filtered probability space $(\Om,\calF,\{\calF_t\},\PP)$, for which expectation is denoted by $\E$. This completes our definition of a solution to \eqref{51}.

A solution to \eqref{MVSDE} is also assumed to be defined on $(\Om,\calF,\PP)$, and it is implicit that $\oo L$ is a boundary term, i.e., that $t\mapsto \oo L(t)$ lies in $C_+(\R_+,\R_+)$ and $ \int_{[0,\iy)} \oo X(t)d\oo L(t)=0$, and that $\la$ is finite and continuous.

Let
\begin{equation}\label{45}
\mu^n_t=\frac{1}{n}\sum_{i\in[n]}\del_{(X^n_i(t),L^n_i(t))},
\qquad
\la^n(t)=\frac{1}{n}\sum_{i\in[n]}L^n_i(t).
\end{equation}
Our first main result is as follows.

\begin{theorem}\label{th1} Let $a\in(-1,\iy)$.

i. There exists a pathwise unique strong solution $(\oo X,\oo L,\la)$ to \eqref{MVSDE}.

ii.
Let $k\in\N$. Then $(X^n_i,L^n_i)_{i\in[k]}\To(\oo X_i,\oo L_i)_{i\in[k]}$ in $C^{(2k)}$, where the latter tuple is given by $k$ independent copies of $(\oo X,\oo L)$. Moreover, $(\mu^n,\la^n)\to(\mu,\la)$ in $C(\R_+,\calP(\R_+^2))\times C^{(1)}$, in probability, where $\mu_t=\PP\circ (\oo X(t),\oo L(t))^{-1}$, $t\ge0$, and $\la$ is as in part (i), otherwise expressed as $\la(t)=\int_{\R_+^2}y\mu_t(dx,dy)$, $t\ge0$.
\end{theorem}

\subsection{Random reflection}

For this setting, the environment is represented by a probability space $(\Xi,\calA,Q)$ on which are defined, for each $n\ge2$, a tuple $\{\rho^n_{ij}:i,j\in[n], i\ne j\}$, of $n(n-1)$ real-valued random variables. They are used for defining the random reflection matrix
\[
\hat R^{(n)}_{ij}=\begin{cases}1 & i=j \\ \frac{\rho^n_{ij}}{n-1} & i\ne j.
\end{cases}
\]

\begin{assumption}\label{assn2}
i.
For some $\eps_\rho>0$ and every $\xi\in\Xi$, $n\ge2$, $i,j\in[n]$, $i\ne j$, one has $|\rho^n_{ij}(\xi)|\le 1-\eps_\rho$.

ii.
The mean $a:=\int\rho^n_{ij}(\xi)Q(d\xi)$ does not depend on $n$, $i$, $j$.

iii. For every $n$ and $i\in[n]$, the $n-1$ random variables $\rho^n_{ij}$, $j\in[n]$, $j\ne i$, are mutually independent under $Q$.
\end{assumption}

Under Assumption \ref{assn2}, SRBM is well defined, for it is easy to see, along the lines of the proof of Lemma \ref{lem01}, that the completely-$\calS$ condition holds for all $\xi$. However, we will rely on a stronger result, which guarantees existence of strong solutions.

Let $\calR^{(n)}$ denote the set of $n\times n$ matrices of the form $R=I+A$, with $I=\mathbf{I}_n$ the identity matrix and $A$ a matrix such that $|A|$, the matrix whose entries are absolute values of those of $A$, has spectral radius $<1$.

\begin{theorem}[\cite{har-rei, wil98inv}]
\label{th-hr}
For $n\in\N$, there exists a map $\Gamma^{(n)}:C_{S_n}(\R_+,\R^n)\times\calR^{(n)}\to C^{(2n)}$ such that whenever
$(\hat\Xi,\hat\calA,\hat Q)$ is a probability space supporting $X_0\sim\nu^{(n)}$, $W$ an $(n,b,\sig)$-BM independent of $X_0$, and $\{\hat\calA_t\}$ is the filtration generated by $X_0+W(t)$, $t\in\R_+$, then, upon setting $(X,L)=\Gamma^{(n)}(X_0+W;R)$, the tuple $(X_0,W,X,L)$ satisfies all requirements in Definition \ref{def1}, but with the filtered probability space being $(\hat\Xi,\hat\calA,\{\hat\calA_t\},\hat Q)$. Moreover, $\Gamma^{(n)}$ is jointly measurable in both arguments.
\end{theorem}
\proof
When all entries of $A$ are nonpositive, the first assertion of the theorem is found in \cite{har-rei}. For signed $A$, it was noticed in \cite[Section 7]{wil98inv} that the proof from \cite{har-rei} continues to hold. As for joint measurability, this follows from the fact that $\Gamma^{(n)}$ (precisely, its second component) was constructed in \cite{har-rei} as the unique fixed point of a contraction $\pi=\pi(\cdot\,;R)$ on $C([0,T],\R^n)$, which itself is jointly measurable on $C([0,T],\R^n)\times\calR^{(n)}$.
\qed

Let
\[
(\hat\Xi,\hat\calA):=\Big(\prod_{n\ge2}S^{(n)}\times C^{(3n)},\prod_{n\ge2}\calS^{(n)}\otimes\calC^{(3n)}\Big),
\]
and denote the canonical process by $(\hat X^n_0,\hat W^n,\hat X^n,\hat L^n)_{n\ge2}$.
For every $\xi\in\Xi$, let $Q_\xi$ be a probability measure on $(\hat\Xi,\hat\calA)$ under which, for each $n$, $\hat X^n_0\sim\nu^{(n)}$, $\hat W^n$ is an $(n,b,\sig)$-BM independent of $\hat X^n_0$, and $Q_\xi$-a.s.,
\[
(\hat X^n,\hat L^n)=\Gamma^{(n)}(\hat X^n_0+\hat W^n;\hat R^{(n)}(\xi)).
\]
Above, $\Gamma^{(n)}(\,\cdot\,; \hat R^{(n)}(\xi))$ is well-defined for every $\xi$, owing to Assumption \ref{assn2}(i) which implies that, for every $\xi\in\Xi$, $\hat R^{(n)}(\xi)\in\calR^{(n)}$.
In view of the first assertion of Theorem \ref{th-hr}, the tuple $(\hat X^n_0,\hat W^n,\hat X^n,\hat L^n,\rho^n(\xi))$ satisfies \eqref{01b} $Q_\xi$-a.s., for every $\xi$. The measure $Q_\xi$ thus defined is called the {\it quenched} measure.

Next, let $(\mathbb{X},\mathbb{A})=(\Xi\times\hat\Xi,\calA\otimes\hat\calA)$. Note that $\xi\mapsto Q_\xi(\hat A)$ is $\calA$-measurable for every $\hat A\in\hat\calA$, owing to the joint measurability of $\Gamma^{(n)}$ stated in Theorem \ref{th-hr}. Hence one can define the {\it annealed} measure $\mathbb{Q}$ as the probability measure on $(\mathbb{X},\mathbb{A})$ satisfying
\begin{equation}\label{x9}
\mathbb{Q}(A\times\hat A)=\int_A Q_\xi(\hat A) Q(d\xi),
\qquad A\in\Xi,\,\hat A\in\hat\Xi.
\end{equation}
On this space, we let $H(\xi,\hat\xi)=H(\hat\xi)$ for $H=\hat X^n_0$, $\hat W^n$, $\hat X^n$ and $\hat L^n$, whereas $\rho^n(\xi,\hat\xi)=\rho^n(\xi)$.

Let $\hat\mu^n$ and $\hat\la^n$ be defined analogously to \eqref{45}, with $(\hat X^n,\hat L^n)$ replacing $(X^n,L^n)$. Our second main result is the following.

\begin{theorem}\label{th2}
Let Assumption \ref{assn2} hold. Let $k\in\N$. Let $(\oo X_i,\oo L_i)$, $\mu$  and $\la$ be as in Theorem~\ref{th1}. Then under $\mathbb{Q}$, as well as for $Q$-a.e.\ $\xi\in\Xi$, under $Q_\xi$, one has $(\hat X^n_i,\hat L^n_i)_{i\in[k]}\To(\oo X_i,\oo L_i)_{i\in[k]}$ in $C^{(2k)}$ and $(\hat\mu^n,\hat\la^n)\to(\mu,\la)$ in $C(\R_+,\calP(\R^2))\times C^{(1)}$ in probability.
\end{theorem}

{\it Remark.}
Whereas working with the strong solution makes the construction of the quenched and annealed measures straightforward, this is not the motivation for imposing the assumption $|\rho^n_{ij}|\le 1-\eps_\rho$. Rather, this condition is required for the proof of Theorem \ref{th2}, which relies on coupling with the dynamics under deterministic reflection. It is in this step that the bound is used in an essential way (see, in particular, the proof of Proposition \ref{prop5}).

{\it Open problem.}
Extend Theorem \ref{th2} by replacing Assumption \ref{assn2}(i), i.e.\ $|\rho^n_{ij}|\le1-\eps_\rho$, with the weaker condition $-1<c_1\le\rho^n_{ij}\le c_2<\iy$.

\section{The case of deterministic reflection coefficients}

Here we present the proof of Theorem \ref{th1}, which proceeds in four steps established by Propositions \ref{prop1}-\ref{prop4} below. The approach is based on the penalty method by which the reflected dynamics are approximated by diffusion in all of $\R^n$ with a penalty term. This makes it possible to take advantage of a classical result for weakly interacting diffusions with a Lipschitz drift, by which the mean field limit exists and is given by the McKean--Vlasov diffusion with penalized dynamics (Proposition \ref{prop2}). Before applying this tool we argue the existence of a unique solution to \eqref{MVSDE} (Proposition \ref{prop1}). We then show that as the penalty parameter diminishes, both the $n$-dimensional system and the one-dimensional McKean--Vlasov diffusion converge to their reflected counterparts (Propositions \ref{prop3} and \ref{prop4}, respectively).

We now state Propositions \ref{prop1}-\ref{prop4}. Their proofs appear in Sections \ref{sec31}-\ref{sec34}, respectively.

\begin{proposition}\label{prop1}
There exists a pathwise unique strong solution to \eqref{MVSDE}.
\end{proposition}

To construct an approximating system via the penalty method, let $Y^n=Y^{\eps,n}$ be defined as the unique strong solution to the following SDE in $\R^n$
\begin{equation}\label{52}
\begin{split}
Y^n_i(t)&=X^n_{0,i}+W^n_i(t)+\La^n_i(t)
+a\lan \La^n(t) \ran_i
\\
\La^n_i(t) &= \int_0^t\frp_\eps(Y^n_i(s))ds
\end{split}
\end{equation}
where, for $\eps>0$, we take $\frp_\eps:\R\to[0,\eps^{-1}]$ to be the Lipschitz function
\[
\frp_\eps(x)=\begin{cases}\eps^{-1} & x\le-\eps\\ -\eps^{-2}x & -\eps<x<0 \\ 0 & x\ge0.
\end{cases}
\]
Note that we have defined $Y^n$ so that it is driven by the same BM as $X^n$ in \eqref{51}, which clearly can be done thanks to the existence of a strong solution to \eqref{52}. The McKean--Vlasov diffusion in this case is given by $\oo Y=\oo Y^\eps$,
\begin{equation}\label{23}
\begin{split}
\oo Y(t)&=\oo X_0+\oo W(t)+\oo\La(t)+a\,\E\oo\La(t)
\\
\oo\La(t)&=\int_0^t\frp_\eps(\oo Y(s))ds.
\end{split}
\end{equation}
The following result is classical; Section \ref{sec32} points to the literature where the proof is found.

\begin{proposition}\label{prop2}
Fix $\eps>0$. There exists a pathwise unique strong solution to \eqref{23}. Moreover, for $k\in\N$, $(Y^n_i,W^n_i)_{i\in[k]}\To(\oo Y_i,\oo W_i)_{i\in[k]}$ in $C^{(2k)}$ as $n\to\iy$, where $(\oo Y_i,\oo W_i)$ are $k$ independent copies of $(\oo Y,\oo W)$.
\end{proposition}

The penalty method, which goes back to \cite{men83}, consists of showing that, for each $n$, the vanishing $\eps$ limit of a penalized diffusion, such as $Y^n$, is given by a reflected diffusion, such as $X^n$. However, here we will need a uniform-in-$n$ convergence.

\begin{proposition}\label{prop3}
There exists $n_0=n_0(a)$ such that, as $\eps\to0$, $(Y^n_1,\La^n_1)\to (X^n_1,L^n_1)$ in $C^{(2)}$ in probability, uniformly in $n\ge n_0$.
\end{proposition}

A similar result is stated now about $\oo X$ and $\oo Y$, of \eqref{MVSDE} and \eqref{23}. These processes are driven by the same $(\oo X_0,\oo W)$, which again can be done thanks to the existence of strong solutions.

\begin{proposition}\label{prop4}
As $\eps\to0$, $(\oo Y^\eps,\oo\La^\eps)\to(\oo X,\oo L)$ in $C^{(2)}$ in probability.
\end{proposition}

\noi{\bf Proof of Theorem \ref{th1}.}
Statement (i) is already established in Proposition \ref{prop1}. For statement (ii), fix $k$ and denote
\[
\bX^n_{[k]}=(X^n_i,L^n_i,W^n_i)_{i\in[k]},
\qquad
\oo\bX_{[k]}=(\oo X_i,\oo L_i,\oo W_i)_{i\in[k]},
\]
\[
\bY^{\eps,n}_{[k]}=(Y^{\eps,n}_i,\La^{\eps,n}_i,W^n_i)_{i\in[k]},
\qquad
\oo\bY^\eps_{[k]}=(\oo Y^\eps_i,\oo\La^\eps_i,\oo W_i)_{i\in[k]},
\]
where $(\oo Y^\eps_i,\oo\La^\eps_i,\oo W_i)_{i\in[k]}$ (respectively, $(\oo X_i,\oo L_i,\oo W_i)_{i\in[k]}$) are $k$ independent copies of $(\oo Y^\eps,\oo\La^\eps,\oo W)$ (respectively $(\oo X,\oo L,\oo W)$). We emphasize that the systems are driven by the same $(\oo X_{0,i},\oo W_i)$.
We now show that $\bX^n_{[k]} \To \oo\bX_{[k]}$ in $C^{(3k)}$. It suffices to show that, given $T$, $\E \psi(\bX^n_{[k]})\to\E \psi(\oo\bX_{[k]})$ for all $\psi\in C_b(C([0,T],\R^{3k}),\R)$. For such $\psi$, write
\[
|\E \psi(\bX^n_{[k]})-\E \psi(\oo\bX_{[k]})|
\le |\E \psi(\bX^n_{[k]})-\E \psi(\bY^{\eps,n}_{[k]})|+|\E \psi(\bY^{\eps,n}_{[k]})-\E \psi(\oo\bY^\eps_{[k]})|+|\E \psi(\oo\bY^\eps_{[k]})-\E \psi(\oo\bX_{[k]})|.
\]
Due to exchangeability, Proposition \ref{prop3} implies the convergence in probability, as $\eps\to0$, of $\bY^{\eps,n}_{[k]}\to\bX^n_{[k]}$. Likewise, Proposition \ref{prop4} implies the convergence in probability, as $\eps\to0$, of $\oo\bY^{\eps}_{[k]}\to\oo\bX_{[k]}$.
Hence, given $\eta>0$, one can find $\eps$ such that the first and third terms on the right of the above display are less than $\eta/3$ each, for all $n\ge n_0(a)$. For that $\eps$, by Proposition \ref{prop2}, the second term is less than $\eta/3$ for all large $n$. We obtain that $|\E \psi(\bX^n_{[k]})-\E \psi(\oo\bX_{[k]})|<\eta$ for all large $n$, and the claim is proved.

Next we show that $\mu^n\to\mu$. Having established propagation of chaos, this statement follows from the general theory. In particular, fix $T$ and let $m$ be the law induced by $(\oo X,\oo L)$ on $C([0,T],\R^2)$. Let $Z^n_i=\{(X^n_i(t),L^n_i(t)),t\in[0,T]\}$. Then
\[
m^n:=\frac{1}{n}\sum_{i\in[n]}\del_{Z^n_i}\to m
\]
in $\calP(C([0,T],\R^2))$ in probability; see \cite[Definition 2.1 and Proposition 2.2]{sznitman1991topics}. Metrizing $\calP(C([0,T],\R^2))$ and $\calP(\R^2)$ by the Levy--Prohorov metric $d_{\rm LP}$ and $\oo d_{\rm LP}$, respectively, it follows from the definition of this metric that, for $t\in[0,T]$,
\[
\oo d_{\rm LP}(\mu^n_t,\mu_t)=\oo d_{\rm LP}(m^n\circ x(t)^{-1},m\circ x(t)^{-1})\le d_{\rm LP}(m^n,m),
\]
with $x=\{x(t)\}$ the canonical process on $C([0,T],\R^2)$.
Taking supremum over $t\in[0,T]$ proves $\mu^n\to\mu$ in $C(\R_+,\calP(\R^2))$ in probability.

Finally we prove the assertion regarding convergence of $\la^n$. If $\la$ is the last component of the solution of \eqref{MVSDE} then by definition $\la(t)=\E\oo L(t)$. Also by definition, the law of $\oo L(t)$ is the second marginal of $\mu_t$, $t\in\R_+$. Hence $\la(t)=\int_{\R_+^2}y\mu_t(dx,dy)$.

To show that $\la^n\to\la$, note that $(X^n_1,X^n_{0,1},W^n_1,L^n_1)\To(\oo X,\oo X_0,\oo W,\oo L)$, as we have shown. Since these two tuples constitute all but the last term in \eqref{51} and respectively \eqref{MVSDE}, it follows that the same statement must hold for the respective last terms as well: $\lan L^n\ran_1\to\la$ in $C^{(1)}$ in probability. Since $L^n_1$ are tight, it follows that $\lan L^n\ran\to\la$ as well.
\qed

\noi{\bf About notation.}
In what follows, for the most part, the dependence of processes on $n$ will be suppressed from the notation in the proofs of lemmas and propositions (for example, we will write $W$ for $W^n$).

\subsection{Proof of Proposition \ref{prop1} (existence and uniqueness for \eqref{MVSDE})}
\label{sec31}

Here we prove existence and uniqueness for the nonlinear RBM. The existence statement is proved based on tightness of the processes $(X^n_1,L^n_1)$, a property used again later in the convergence proof.
It is worth mentioning that, in the case where the initial conditions $X^n_{0,i}$ are i.i.d., the results of this subsection alone give the convergence $(X^n_1,L^n_1)\To(\oo X,\oo L)$.

An elementary lemma about the Skorohod problem is as follows.
\begin{lemma}\label{lem02}
Let $w\in C_{\R_+}(\R_+,\R)$, $x\in C(\R_+,\R_+)$ and $\ell\in C_+(\R_+,\R_+)$ be such that $x=w+\ell$ and $\int_{[0,\iy)}xd\ell=0$. Then
\[
\ell(t)=\sup_{s\le t}(w(s)^-), \qquad t\ge0.
\]
\end{lemma}
\proof See \cite[Section 8]{chu-wil}.
\qed

\begin{lemma}\label{lem03}
For $a>-1$, $\{X^n_1,n\ge2\}$ and $\{L^n_1,n\ge2\}$ are tight in $C(\R_+,\R_+)$.
Moreover, $\sup_n\E[L^n_1(t)^2]<\iy$, for every $t$.

\end{lemma}

\proof
Using Lemma \ref{lem02} in \eqref{51},
\begin{equation}\label{06}
L_i(t)=\sup_{s\le t}[(X_{0,i}+W_i(s)+a\lan L(s)\ran_i)^-].
\end{equation}
Consider first $a\ge0$. Then
\begin{equation}\label{r1}
L_i(t)\le\sup_{s\le t}W_i(s)^-\le\|W_i\|_t.
\end{equation}
Moreover, for $t_1<t_2$, $L_i(t_2)-L_i(t_1)\le |W_i(t_2)-W_i(t_1)|$ hence
\begin{equation}\label{r02}
w_T(L_i,\del)\le w_T(W_i,\del).
\end{equation}
These two estimates show that, for fixed $i$, $L_i$, as a sequence indexed by $n$, are tight in $C(\R_+,\R_+)$. In view of this, to show now that, for fixed $i$, $X_i$ are tight, it suffices to show that $\lan L\ran_i$ are. But
\[
\E\lan L(t)\ran_i=\E L_1(t)\le \E\|W_1\|_t
\]
\[
\E w_T(\lan L\ran_i,\del)\le\E w_T(L_1,\del)\le \E w_T(W_1,\del).
\]
Thus, given $\eps$,
\[
\PP(w_T(\lan L\ran_i,\del)>\eps)
\le \eps^{-1}\E w_T(W_1,\del)\le c\eps^{-1}(\del\log(T/\del))^{1/2},
\]
for some constant $c$ independent of $n$. This gives tightness of $\lan L\ran_i$, and in turn, of $X_1$.

Next consider $-1<a<0$. Again by \eqref{06},
\[
L_i(t)\le \|W_i\|_t+|a|\lan L(t)\ran_i
\]
and averaging over $i$ gives
\begin{equation}
\label{r2}
\lan L(t)\ran\le c_a\lan \|W\|_t\ran,
\end{equation}
where $c_a=(1-|a|)^{-1}$,
and so $\E L_1(t)=\E\lan L\ran(t)\le c_a\E\|W_1\|_t$.
Similarly, by \eqref{06}
\begin{equation}
\label{r3}
w_T(L_i,\del)\le w_T(W_i,\del)+|a|w_T(\lan L\ran_i,\del)
\end{equation}
and, because $w_T(\lan L\ran_i,\del)\le \lan w_T(L,\del)\ran_i$, we have
\[
\E w_T(L_1,\del)\le \E w_T(W_1,\del)+|a|\E w_T(L_1,\del)
\]
and once again the tightness of $L_1$ follows.
The tightness of $\lan L\ran_i$ and consequently of $X_i$ now follows as in the previous case.

The second moment estimate follows from \eqref{r1} in the case $a\ge0$. In the case $-1<a<0$, we have using \eqref{r2},
\[
\lan L(t)\ran_i \le 2\lan L(t)\ran\le 2c_a\lan \|W\|_t\ran.
\]
Using this in \eqref{06},
\begin{equation}\label{x1}
L_i(t)\le\|W_i\|_t+|a|\lan L(t)\ran_i\le\|W_i\|_t+2|a|c_a\lan\|W\|_t\ran,
\end{equation}
and the claim follows.
For later use let us record the fact that, similarly to the argument leading to \eqref{x1}, one gets from \eqref{r3}
\begin{equation}\label{x2}
w_t(L_i,\del)\le w_t(W_i,\del)+2|a|c_a\lan w_t(W,\del)\ran.
\end{equation}
\qed

\noi {\bf Proof of Proposition \ref{prop1}.}
First we show existence. To this end, consider a special case of our setting, in which $X_{0,i}$ are i.i.d.\ across $i$ for each $n$ (and distributed $\mu_0$). This gives independence of $L_i$ for each $n$.
Write \eqref{51} as
\begin{equation}\label{46}
X_i(t)=X_{0,i}+W_i(t)+L_i(t)+a\E[L_i(t)]+a\lan L(t)-\E[L(t)]\ran_i.
\end{equation}
Consider the sequence $(X_1,L_1,W_1)$, $n\ge2$, proved in Lemma \ref{lem03} to be tight, and let $(\oo X,\oo L,\oo W)$ be a subsequential weak limit.
Using independence and the second moment bound of Lemma \ref{lem03},
\[
{\rm Var}\lan L(t)\ran_i=(n-1)^{-1}{\rm Var}(L_1(t))\le c(n-1)^{-1},
\]
showing that the last term in \eqref{46} goes to zero in probability for each $t$.
Let $\la(t)=\E\oo L(t)$. Then by uniform integrability of $L_i$, $\E[L_i]\to\la$ for each $t$.
By tightness of $(X_i,X_{0,i}+W_i,L_i)$, $t\mapsto\E[L_i(t)]$ must converge to a continuous function. Hence $\E[L_i]\to\la$ in $C^{(1)}$. As a result, 
\[
\oo X(t)=\oo X_0+\oo W(t)+\oo L(t)+a\la(t).
\]
The boundary property $\int\oo Xd\oo L=0$ follows from $\int X^n_1dL^n_1=0$. This shows that $(\oo X,\oo L,\oo W,\la)$ is a weak solution of \eqref{MVSDE}. To produce a strong solution, given a $(b,\sig)$-BM $\tilde W$ and an initial condition $\tilde X_0$, use the function $\la$ just constructed to define $\tilde L(t)=\sup_{s\le t}[(\tilde X_0+\tilde W(s)+a\la(s))^-]$ and $\tilde X=\tilde X_0+\tilde W+\tilde L+a\la$. It is clear that $(\tilde X,\tilde L,\tilde W)$ is equal in law to $(\oo X,\oo L,\oo W)$, and thus the former is a strong solution.

For pathwise uniqueness, let $\oo X_0$ and $\oo W$ be given, and let $(\oo X,\oo L,\oo\la)$ and $(\tilde X,\tilde L,\tilde\la)$ be two corresponding solutions. Denote $V=\oo X-\tilde X$. Then
\begin{equation}\label{47}
v(t):=\E V(t)=\E[\oo L(t)+a\oo\la(t)-\tilde L(t)-a\tilde\la(t)] = (1+a)(\oo\la(t)-\tilde\la(t)).
\end{equation}
Thus
\[
V(t)^2=2\int_0^tV(s)(d\oo L(s)+ad\oo\la(s)-d\tilde L(s)-ad\tilde\la(s)),
\]
and because $\oo X=0$ $d\oo L$-a.e.\ while $\tilde X=0$ $d\tilde L$-a.e., we have
\[
V(t)^2\le 2a\int_0^tV(s)(d\oo\la(s)-d\tilde\la(s))
=\frac{2a}{1+a}\int_0^tV(s)dv(s),
\]
where we used \eqref{47}. This shows that $\E[V(t)^2]<\iy$. Moreover,
\[
\E[V(t)^2]\le\frac{2a}{1+a}\int_0^tv(s)dv(s)
=\frac{a}{1+a}v(t)^2.
\]
For $a\le 0$ this shows that $\E[V(t)^2]=0$ for all $t$. For $a>0$, $\E[V(t)^2]\le \frac{a}{1+a}\E[V(t)^2]$ and again $\E[V(t)^2]=0$ for all $t$. This shows $\oo X=\tilde X$ a.s.

Next, because $\oo L(t)$ has finite expectation, in follows from \eqref{MVSDE} that so does $\oo X(t)-\oo X_0$. Moving $\oo X_0$ to the left side of \eqref{MVSDE} and taking expectation on both sides shows that $\oo\la=\tilde\la$. This in turn implies $\oo L=\tilde L$ a.s., completing the proof.
\qed

\subsection{Proof of Proposition \ref{prop2} (propagation of chaos for the penalized system)}
\label{sec32}
For strong existence and pathwise uniqueness see \cite[Theorem 1.1]{sznitman1991topics}.
Next, for the $n\to\iy$ result, note that \eqref{52} can be written as
\[
Y^n_i(t)=X^n_{0,i}+W^n_i(t)+\int_0^t\frp_\eps(Y^n_i(s))ds+a\int_0^t\lan \frp_\eps(Y^n(s))\ran ds+e_{i,n}(t)
\]
where the error term is bounded by $c\eps^{-1}tn^{-1}$, and so the large $n$ asymptotics is the same as that without the error term. Also note that $(Y^n,W^n)$ can be viewed as a system of $n$ particles in $\R^2$ with a bounded, continuous drift coefficient. For this dynamics we can apply classical results in the field. Some authors require moment assumptions on the initial condition; the case without moment assumptions is covered in \cite{gar88}. In particular, the convergence in law of the $k$-tuple appears in \cite[Theorem 4.1]{gar88} and its corollary, where \cite[Section 5.2]{gar88} relaxes the moment conditions on $X_{0,i}$ at the price of assuming bounded coefficients, which holds in our case. This completes the proof of Proposition \ref{prop2}.
\qed

\subsection{Proof of Proposition \ref{prop3} (penalized system as $\eps\to0$)}
\label{sec33}

\begin{lemma}\label{lem07}
There exists $n_0=n_0(a)$ such that for every $\eps\in(0,1)$, $n\ge n_0$ and $t\ge0$,
\[
\E[(X^n_1(t)-Y^n_1(t))^2]
\le 6(1+a)\{\E[(\|Y_1^{n,-}\|_t)^2]\E[L^n_1(t)^2]\}^{1/2}.
\]
\end{lemma}
\proof
Write the two systems as
\begin{equation}\label{22}
\begin{split}
X_i(t)&=X_{0,i}+W_i(t)+\al_nL_i(t)+a\beta_n\lan L(t)\ran
\\
Y_i(t)&=X_{0,i}+W_i(t)+\al_n\int_0^t\frp_\eps(Y_i(s))ds+a\beta_n\int_0^t\lan\frp_\eps(Y(s))\ran ds,
\end{split}
\end{equation}
where $\al_n=1-\frac{a}{n-1}$, $\beta_n=\frac{n}{n-1}$. Then
\begin{align*}
(X_i(t)-Y_i(t))^2
&= 2\int_0^t(X_i-Y_i)[\al_ndL_i-\al_n\frp_\eps(Y_i)ds +a\beta_nd\lan L\ran-a\beta_n\lan \frp_\eps(Y)\ran ds]
\\
&\le 2\al_n\int_0^t Y_i^-dL_i
+2a\beta_n\int_0^t(X_i-Y_i)[d\lan L\ran-\lan \frp_\eps(Y)\ran ds].
\end{align*}
Averaging,
\begin{align*}
\lan (X(t)-Y(t))^2\ran &\le 2\al_n\Big\lan\int_0^t Y^-dL \Big\ran
+2a\beta_n\int_0^t\lan X-Y\ran[d\lan L\ran-\lan \frp_\eps(Y)\ran ds].
\end{align*}
Note that by \eqref{22},
\[
\begin{split}
\lan X(t)\ran&=\lan X_0+W(t)\ran +(\al_n+a\beta_n)\lan L(t)\ran
\\
\lan Y(t)\ran&=\lan X_0+W(t)\ran+(\al_n+a\beta_n)\int_0^t\lan\frp_\eps(Y)\ran ds
\end{split}
\]
which gives
\[
d\lan L\ran-\lan\frp_\eps(Y)\ran dt=\frac{d\lan X-Y\ran}{\al_n+a\beta_n}.
\]
Because the quadratic variation of $\lan X-Y\ran$ vanishes, $\int_0^t\lan X-Y\ran d\lan X-Y\ran=\frac{1}{2}\lan X-Y\ran^2(t)$, hence
\[
\lan X(t)-Y(t)\ran^2\le 2\al_n\Big\lan \int_0^tY^-dL\Big\ran+\frac{a\beta_n}{\al_n+a\beta_n}\lan X(t)-Y(t)\ran^2.
\]
Note that $2\al_n\to2$ and $1-\frac{a\beta_n}{\al_n+a\beta_n}\to\frac{1}{1+a}\in(0,\iy)$. Let $n_0=n_0(a)$ be such that for all $n\ge n_0$, $2\al_n<3$ and $1-\frac{a\beta_n}{\al_n+a\beta_n}>\frac{1}{2(1+a)}$. Then for $n\ge n_0$,
\[
\lan X(t)-Y(t)\ran^2\le 6(1+a)\Big\lan \int_0^t Y^-dL\Big\ran.
\]
By exchangeability,
\[
\E[(X_1(t)-Y_1(t))^2]\le 6(1+a)\,\E\int_0^tY^-_1dL_1
\le 6(1+a)\{\E[(\|Y_1^-\|_t)^2]\E[L_1(t)^2]\}^{1/2}.
\]
\qed

\begin{lemma}\label{lem05}
For every $t$, $\lim_{\eps\to0}\sup_n\E[\sup_{s\in[0,t]}(Y^n_1(s)^-)^2]=0$.
\end{lemma}

\proof
Fix $t$.
Denoting $A=A(\eps,n)=\sup_{[0,t]}Y^n_1(s)^-$, write
\begin{equation}\label{41}
\E[A^2]=\int_0^\iy\PP(A^2>\eta)d\eta\le\eta_0+\int_{\eta_0}^\iy\PP(A>\eta^{1/2})d\eta.
\end{equation}
For $\eta$ such that $\frac{1}{2}\eta^{1/2}>\eps$, on the event $A>\eta^{1/2}$, there exist $0\le\sig<\tau\le t$ such that $Y_1(\sig)=-\frac{1}{2}\eta^{1/2}$, $Y_1(\tau)=-\eta^{1/2}$ and $Y_1\le-\eps$ on $[\sig,\tau]$. Using this in \eqref{52}, we have on the same event
\begin{align*}
-\frac{1}{2}\eta^{1/2}&=Y_1(\tau)-Y_1(\sig)\\
&= W_1(\tau)-W_1(\sig)+\eps^{-1}(\tau-\sig)+a\int_\sig^\tau\lan\frp_\eps(Y(s))\ran_ids\\
&\ge W_1(\tau)-W_1(\sig)+C_a\eps^{-1}(\tau-\sig),
\\
C_a&:=1+(a\w 0)\in(0,1].
\end{align*}
Let $\kap>0$. Then, on the event $A>\eta^{1/2}$, the condition $\tau-\sig>\kap$ implies $0\ge W_1(\tau)-W_1(\sig)+C_a\eps^{-1}\kap+\eta^{1/2}$,
hence $2\|W_1\|_t\ge C_a\eps^{-1}\kap+\eta^{1/2}$; whereas the condition
$\tau-\sig\le\kap$ implies $w_t(W_1,\kap)\ge\eta^{1/2}$.
With $c$ a positive constant (depending on $t$) whose value may change from one appearance to another, one has
\begin{equation}\label{b20}
\PP(\|W\|_t>x)\le ce^{-cx^2},
\qquad
\PP(w_t(W,\del)>x)\le c\del^{-1}e^{-c x^2\del^{-1}},
\qquad
x>0.
\end{equation}
Combining, we have
\[
\PP(A>\eta^{1/2})\le ce^{-c(\eps^{-1}\kap+\eta^{1/2})^2}+c\kap^{-1}e^{-c\eta\kap^{-1}}
\le
ce^{-c\eps^{-2}\kap^2-c\eta}+c\kap^{-1}e^{-c\eta\kap^{-1}}.
\]
Using this in \eqref{41} gives
\[
\E[A^2]\le\eta_0+ce^{-c\eps^{-2}\kap^2}+ce^{-c\eta_0\kap^{-1}}.
\]
Choosing $\kap=\eps^{1/2}$ and $\eta_0=\eps^{1/4}$, the expression on the right, which does not depend on $n$, goes to $0$ as $\eps\to0$.
\qed

\begin{lemma}\label{lem06}
For every $t>0$ and $\theta>0$ there exist $\del>0$ and $\eps_0>0$ such that
\[
\sup_{0<\eps<\eps_0}\sup_n\PP(w_t(Y^n_1,\del)\vee w_t(\La^n_1,\del)>\theta)<\theta.
\]
\end{lemma}

\proof
To estimate $w_t(\La_1,\del)$, let $0\le\sig\le\tau\le t$ be such that $\tau-\sig\le\del$ and
$\La_1(\tau)-\La_1(\sig)=w_t(\La_1,\del)$. Since $\La_1$ increases only when $Y_1\le0$, we may assume without loss of generality that, at both $\sig$ and $\tau$, $Y_1\le0$. Thus by \eqref{52},
\begin{align*}
w_t(\La_1,\del) &= Y_1(\tau)-Y_1(\sig)-W_1(\tau)+W_1(\sig)-a\lan \La(\tau)-\La(\sig)\ran_1
\\
&\le \|Y_1^-\|_t+w_t(W_1,\del)+a^-w_t(\lan \La\ran_1,\del)
\\
&\le \|Y_1^-\|_t+w_t(W_1,\del)+a^-\lan w_t(\La,\del)\ran_1.
\end{align*}
With $c=1-a^-\in(0,1]$, this gives
\[
c\,\E[w_t(\La_1,\del)]\le\E[\|Y^-_1\|_t]+\E w_t(W_1,\del),
\]
and the estimate on $w_t(\La_1,\del)$ follows by Chebychev's inequality and Lemma \ref{lem05}.

As for the estimate on $w_t(Y_1,\del)$,
using \eqref{52} and again the inequality $w_t(\lan\La\ran_1,\del)\le \lan w_t(\La,\del)\ran_1$, we get
\[
\E w_t(Y_1,\del)\le \E w_t(W_1,\del)+(1+|a|)\E w_t(\La_1,\del),
\]
and the estimate follows from the one on $w_t(\La_1,\del)$.
\qed

\noi{\bf Proof of Proposition \ref{prop3}.}
If we combine Lemma \ref{lem07}, Lemma \ref{lem05} and the moment bound from Lemma \ref{lem03}, we get that
\begin{equation}\label{b1}
\lim_{\eps\to0}\sup_{n\ge n_0}\E[(X^n_1(t)-Y^n_1(t))^2]=0.
\end{equation}
This gives $Y^n_1\to X^n_1$ pointwise in probability, as $\eps\to0$, uniformly in $n\ge n_0$. To strengthen this to convergence in $C^{(1)}$, it suffices to show that given $t>0$ and $\theta>0$ there exist $\del>0$ and $\eps_0>0$ such that $\sup_{0<\eps<\eps_0}\sup_n\PP(w_t(Y^n_1-X^n_1,\del)>\theta)<\theta$. This follows from the tightness of $X^n_1$ proved in Lemma \ref{lem03} and the estimate from Lemma \ref{lem06}.

Next, toward showing $\La^n_1\to L^n_1$, note by averaging in \eqref{51} and \eqref{52},
\[
\lan X-Y\ran = (1+a)\lan L-\La\ran,
\]
hence, with $c=(1+a)^{-1}$,
\[
|\lan L\ran-\lan\La\ran| \le c\lan |X-Y|\ran.
\]
Recall the $(\al_n,\beta_n)$ representation from \eqref{22}. Then, with $c_n=|a|\beta_n\le2|a|$,
\[
\al_n|L_1-\La_1|\le |X_1-Y_1|+c_n|\lan L\ran-\lan\La\ran|\le |X_1-Y_1|+2c|a|\lan |X-Y|\ran,
\]
and because $\al_n>1/2$ for large $n$,
\[
\E|L_1(t)-\La_1(t)|\le 2(1+c|a|)\E|X_1(t)-Y_1(t)|.
\]
In view of \eqref{b1},
\[
\lim_{\eps\to0}\sup_{n\ge n_0}\E|L^n_1(t)-\La^n_1(t)|=0,
\]
and the pointwise convergence in probability of $\La^n_1\to L^n_1$ as $\eps\to0$, uniformly in $n\ge n_0$, follows. With the help of Lemma \ref{lem06}, one deduces convergence in $C^{(1)}$ in exactly the same way as in the argument above for $Y^n_1\to X^n_1$.
\qed

\subsection{Proof of Proposition \ref{prop4} (nonlinear diffusion as $\eps\to0$)}
\label{sec34}

We have
\begin{align*}
(\oo X(t)-\oo Y(t))^2
&=2\int_0^t(\oo X-\oo Y)[d\oo L-\frp_\eps(\oo Y)ds+a\,d\,\E\oo L-a\,\frp_\eps(\oo Y)ds]
\\
&\le 2\int_0^t\oo Y^-d\oo L+2a\int_0^t(\oo X-\oo Y)(d\,\E\oo L-\E\frp_\eps(\oo Y)ds).
\end{align*}
Now,
\[
\E(\oo X(t)-\oo Y(t))=(1+a)\Big[\E\oo L(t)-\int_0^t\E\frp_\eps(\oo Y)ds\Big].
\]
Hence
\[
(\oo X(t)-\oo Y(t))^2 \le 2\int_0^t\oo Y^-d\oo L + \frac{2a}{1+a}\int_0^t(\oo X-\oo Y)d\,\E(\oo X-\oo Y).
\]
Denoting $\Del(t)=\E[(\oo X(t)-\oo Y(t))^2]$,
\begin{align*}
\Del(t) &\le 2\,\E\int_0^t\oo Y^-d\oo L+\frac{2a}{1+a}\int_0^t\E(\oo X-\oo Y)d\,\E(\oo X-\oo Y)
\\
&=2\,\E\int_0^t\oo Y^-d\oo L+\frac{a}{1+a}\Del(t)
\end{align*}
thus
\begin{equation}\label{b2}
\Del(t)\le 2(1+a)\E\int_0^t\oo Y^-d\oo L \le c\{\E[(\|\oo Y^-\|_t)^2]\E[\oo L(t)^2]\}^{1/2}.
\end{equation}
To show that for every $t$, $\Del(t)\to0$ as $\eps\to0$, observe by Lemma \ref{lem02} that
\[
\oo L(t)=\sup_{s\le t}[(\oo X_0+\oo W(s)+a\la(s))^-]
\le \sup_{s\le t}[(\oo W(s)+a\la(s))^-]\le \|\oo W\|_t+|a|\la(t).
\]
Hence $\E[\oo L(t)^2]<\iy$. Moreover, $\E[(\|\oo Y^-\|_t)^2]\to0$ as $\eps\to0$, which is proved along the lines of the proof of Lemma \ref{lem05}; we skip the details. By \eqref{b2}, $\Del(t)\to0$ as $\eps\to0$, for every $t$.

For every $t$ and $\theta>0$ there exist $\del>0$ and $\eps_0>0$ such that
\begin{equation}\label{b3}
\sup_{0<\eps<\eps_0}\PP(w_t(\oo Y,\del)\vee w_t(\oo\La,\del)>\theta)<\theta.
\end{equation}
Again, the proof of this fact is very similar to that of Lemma \ref{lem06} and we skip the details.
Given the pointwise convergence $\Del(t)\to0$, \eqref{b3} establishes that $\oo Y\to\oo X$ in $C^{(1)}$ in probability, as $\eps\to0$.

Next, by \eqref{MVSDE} and \eqref{23}, with $c=(1+a)^{-1}\in(0,\iy)$,
\[
|\E\oo\La(t)-\E\oo L(t)|=c\,|\E(\oo Y(t)-\oo X(t))|
\le c\,\E|\oo Y(t)-\oo X(t)|,
\]
hence
\[
|\oo\La(t)-\oo L(t)|\le|\oo Y(t)-\oo X(t)|+|a|\, |\E\oo\La(t)-\E\oo L(t)|,
\]
and we obtain $\E|\oo\La(t)-\oo L(t)|\le (1+c|a|)\,\E|\oo Y(t)-\oo X(t)|\to0$ as $\eps\to0$ in view of the fact $\Del(t)\to0$. This shows $\oo\La\to \oo L$ pointwise in probability. Convergence in probability in $C^{(1)}$ again follows from \eqref{b3}.
\qed

\section{The case of random reflection coefficients}

Our approach to dealing with random reflection is to couple the SRBM with random reflection to the one with deterministic reflection, and use the established result about the latter.

For $q=Q$ and $\mathbb{Q}$, let $\E^q$ denote expectation w.r.t.\ $q$ on the respective space.
Recall that the tuple $(\hat X^n_0,\hat W^n,\hat X^n,\hat L^n,\rho^n)$ constructed on $(\mathbb{X},\mathbb{A})$ satisfies \eqref{01b}, that is,
\begin{equation}\label{b11}
\hat X^n_i(t)=\hat X^n_{0,i}+\hat W^n_i(t)+\hat L^n_i(t)+\frac{1}{n-1}\sum_{j\ne i}\rho^n_{ij}\hat L^n_j(t),
\qquad i\in[n].
\end{equation}
Recall also that $\E^Q[\rho^n_{ij}]=a$ for all $n,i,j$, and let $R^{(n)}=R^{(n,a)}$ be as in \eqref{05}.
Define on $(\mathbb{X},\mathbb{A})$ the SRBM with deterministic reflection, $(X^n,L^n)=\Gamma^{(n)}(\hat X^n_0+\hat W^n;R^{(n,a)})$. Then
\begin{equation}\label{b12}
X^n_i=\hat X^n_{0,i}+\hat W^n_i(t)+L^n_i(t)+a\lan L^n(t)\ran_i,
\qquad i\in[n],
\end{equation}
and $(X^n_i,L^n_i)$ is equal in law to the tuple previously defined on $(\Om,\calF)$ with the same notation. We continue to use this notation, with a slight abuse. Denote
\[
\Del X^n_i=\hat X^n_i-X^n_i,
\qquad
\Del L^n_i=\hat L^n_i-L^n_i.
\]
The main step toward proving Theorem \ref{th2} is to show that our assumptions on $\rho^n_{ij}$ lead to a very strong form of proximity between the two systems, as follows.
\begin{proposition}\label{prop5}
For every $t>0$ and $x>0$,
\begin{equation}\label{b14}
\mathbb{Q}(\max_{i\in[n]}(\|\Del X^n_i\|_t\vee\|\Del L^n_i\|_t)>x \text{ i.o.})=0.
\end{equation}
\end{proposition}

Before proving this result we strength the estimates from Lemma \ref{lem03}. In what follows, $c$ will denote a positive constant that may depend on $t$, whose value may change from one appearance to another.

\begin{lemma}\label{lem08}
Fix $t$. Then there are constants $c>0$ such that
for every $n$, $i\in[n]$, $\del>0$ and $r>0$,
\[
\mathbb{Q}(L_i(t)>r)\le ce^{-cr^2},
\qquad
\mathbb{Q}(w_t(L_i,\del)>r)\le c\del^{-1}e^{-cr^2\del^{-1}}.
\]
\end{lemma}
\proof
For $a\ge0$, by \eqref{r1} and \eqref{r02},
$L_i(t)\le\|W_i\|_t$ and $w_t(L_i,\del)\le w_t(W_i,\del)$ and so by the estimates \eqref{b20} the claim follows.

For $a\in(-1,0)$, with $c_a=(1-|a|)^{-1}$, we have by \eqref{x1} that $L_i(t)\le\|W_i\|_t+2|a|c_a\lan\|W\|_t\ran$.
Since $\|W_i\|_t$ are i.i.d., we have by a Chernoff bound, $\mathbb{Q}(\lan\|W\|_t\ran-m>r)\le e^{-cnr^2}$, $r>0$, where $m=\E^{\mathbb{Q}}[\|W_1\|_t]$. Hence
\[
\mathbb{Q}(\lan\|W\|_t\ran>r)\le ce^{-cr^2},
\qquad r>0,
\]
and the first estimate follows. Next, by \eqref{x2},
$w_t(L_i,\del)\le w_t(W_i,\del)+2|a|c_a\lan w_t(W,\del)\ran$. In view of \eqref{b20}, it remains to show that $\mathbb{Q}(\lan w_t(W,\del)\ran>z)\le c\del^{-1}e^{-cz^2\del^{-1}}$. Chernoff bound gives,
with $\tilde m=\tilde m(t,\del)=\E^{\mathbb{Q}}[w_t(W_1,\del)]$
\[
\mathbb{Q}(\lan w_t(W,\del)\ran-\tilde m(t,\del)>r)\le e^{-cnr^2\del^{-1}},\qquad r>0.
\]
By the tail bound in \eqref{b20}, $\tilde m(t,\del)\le c(\del\log(t/\del))^{1/2}$. Combining, we get for all $r>0$,
\[
\mathbb{Q}(\lan w_t(W,\del)\ran>r)\le c\del^{-1}e^{-cr^2\del^{-1}}.
\]
\qed

\noi{\bf Proof of Proposition \ref{prop5}.}
By Lemma \ref{lem02},
\[
\hat L_i(t)=\sup_{s\le t}\Big\{\Big[\hat X_{0,i}+\hat W_i(s)+\frac{1}{n-1}\sum_{j\ne i}\rho_{ij} \hat L_j(s)\Big]^- \Big\},
\]
and
\[
L_i(t)=\sup_{s\le t}\{[\hat X_{0,i}+\hat W_i(s)+a\lan L(s)\ran_i]^-\}.
\]
Denote
\[
B_i(t)=\frac{1}{n-1}\sum_{j\ne i}\rho_{ij}(\hat L_j(t)-L_j(t)),
\qquad
C_i(t)=\frac{1}{n-1}\sum_{j\ne i}(\rho_{ij}-a)L_j(t).
\]
Then
\[
\|\Del L_i\|_t\le \|B_i\|_t+\|C_i\|_t
\le(1-\eps_\rho)\lan \|\Del L\|_t\ran_i+\|C_i\|_t.
\]
Hence
\(
\lan\|\Del L\|_t\ran\le\eps_\rho^{-1}\lan\|C\|_t\ran.
\)
For $n$ so large such that $(1-\eps_\rho)n/(n-1)<1$, this gives
\[
\|\Del L_i\|_t\le\eps_\rho^{-1}\lan\|C\|_t\ran+\|C_i\|_t.
\]
Thus
\[
\max_i\|\Del L_i\|_t\le c\max_i\|C_i\|_t.
\]
Moreover, noting that $\Del X_i=\Del L_i+B_i+C_i$,
\(
|\Del X_i|\le |\Del L_i|+\lan|\Del L| \ran_i+|C_i|,
\)
and therefore
\[
\max_i\|\Del X_i\|_t\le c\max_i\|C\|_t.
\]
Hence to complete the proof it suffices to show that, for every $x>0$, $\mathbb{Q}(\max_i\|C_i\|_t>x)$ is summable over $n\ge2$.

To this end, note that, under $\mathbb{Q}$, $\{\rho_{ij}\}$ and $\{L_i\}$ are independent. By Assumption \ref{assn2}, for fixed $n$ and $i$, $\rho_{ij}$, $j\in[n]$ are mutually independent, have mean $a$, and satisfy $|\rho_{ij}-a|\le2$. Fix $t$ and denote $F=F(n,t)=\sig\{L_j(s):s\le t, j\in[n]\}$. By Azuma's inequality, for $s\le t$ and $x>0$,
\[
\mathbb{Q}(|C_i(s)|>x\,|\,F)\le 2\exp\Big\{-\frac{(n-1)^2x^2}{2\sum_{j\ne i}L_j(s)^2}\Big\}
\le2\exp\Big\{-\frac{(n-1)^2x^2}{2\sum_{j\ne i}L_j(t)^2}\Big\}.
\]
Hence, for $y>0$,
\[
\mathbb{Q}(|C_i(s)|>x \text{ and } \max\{L_j(t),j\in[n]\}\le y)\le 2e^{-nx^2/(3y^2)}.
\]
Using now Lemma \ref{lem08},
\[
\mathbb{Q}(|C_i(s)|>x)
\le \mathbb{Q}(\max\{L_j(t),j\in[n]\}> y)+2e^{-nx^2/(3y^2)}
\le cne^{-cy^2}+2e^{-nx^2/(3y^2)}.
\]
On the event $\|C_i\|_t>x$ one must have either $|C_i(k\del)|>x/2$ for some $k\in\N$ such that $k\del\le t$, or $w_t(C_i,\del)>x/2$. Applying the above display to all such $k$, and using Lemma \ref{lem08},
\[
\mathbb{Q}(\|C_i\|_t>x)\le cn\del^{-1}e^{-cy^2}+c\del^{-1}e^{-cnx^2y^{-2}}+c\del^{-1}e^{-cx^2\del^{-1}}.
\]
Choosing $y=y(n)=n^{1/4}$ and $\del=n^{-1}$, the above expression takes the form
\[
cn^2e^{-cn^{1/2}}+cne^{-cn^{1/2}x^2}+cne^{-cnx^2}.
\]
Multiplying by $n$ gives an upper bound on $\mathbb{Q}(\max_i\|C_i\|_t>x)$. The result follows.
\qed

\noi{\bf Proof of Theorem \ref{th2}.}
Under $\mathbb{Q}$, the convergence of $(\hat X^n_i,\hat L^n_i)_{i\in[k]}$ is an immediate consequence of Theorem \ref{th1} and Proposition \ref{prop5}, which show that $(X^n_i,L^n_i)_{i\in[k]}\To (\oo X_i,\oo L_i)_{i\in[k]}$ in $C^{(2k)}$, and $(\Del X^n_i,\Delta L^n_i)_{i\in[k]}\to0$ in $C^{(2k)}$, $\mathbb{Q}$-a.s. The same argument gives the convergence $\hat\la^n\to\la$ in $C^{(1)}$ in probability. As for the convergence $\hat\mu^n\to\mu$, note that Proposition \ref{prop5} implies $\sup_{t\in[0,T]}\calW_1(\hat\mu^n_t,\mu^n_t)\to0$, $\mathbb{Q}$-a.s., with $\calW_1$ the $1$-Wasserstein distance. Since $d_{\rm LP}^2\le \calW_1$, the convergence $\mu^n\to\mu$ stated is Theorem \ref{th1}, gives $\hat\mu^n\to\mu$ in $C(\R_+,\calP(\R^2))$ in probability.

Next consider the statement regarding the quenched measure.
Note first that $(X^n,L^n)$, defined in the beginning of this section as processes on $(\mathbb{X},\mathbb{A})$, can be defined on $(\hat\Xi,\hat\calA)$, and since their construction involves only $\hat X^n_0$, $\hat W^n$ and $R^{(n)}$, their law under $Q_\xi$ is as that under $\PP$. Hence again
\begin{equation}\label{x01}
\text{for $\xi\in\Xi$, under $Q_\xi$,
$(X^n_i,L^n_i)_{i\in[k]}\To (\oo X_i,\oo L_i)_{i\in[k]}$ in $C^{(2k)}$.}
\end{equation}

The statement $(\Del X^n_i,\Del L^n_i)_{i\in[k]}\to0$ $\mathbb{Q}$-a.s.
can also be translated to a statement about $Q_\xi$. That is, if $B\in\calA\otimes\hat\calA$ is a set of full $\mathbb{Q}$-measure, then in view of \eqref{x9}, its sections
\[
B_\xi=\{\hat\xi\in\hat\Xi:(\xi,\hat\xi)\in B\},
\qquad\xi\in\Xi,
\]
satisfy
\[
1=\mathbb{Q}(B)=\int Q_\xi(B_\xi)Q(d\xi),
\]
showing that, for $Q$-a.e.\ $\xi$, $B_\xi$ is of full $Q_\xi$-measure.
As a result, for $Q$-a.e.\ $\xi\in\Xi$, one has $(\Del X^n_i,\Del L^n_i)_{i\in[k]}\to0$ in $C^{(2k)}$, $Q_\xi$-a.s. Combined with \eqref{x01}, this gives the convergence of $(\hat X^n_i,\hat L^n_i)_{i\in[k]}$ under $Q_\xi$ as stated in the theorem.

Finally, the claim regarding the convergence of $(\hat\mu^n,\hat\la^n)$ under $Q_\xi$, based on that under $\mathbb{Q}$, is proved along the same lines.
\qed

\subsection*{Acknowledgment} Research was supported by ISF grant 3240/25.


\bibliographystyle{is-abbrv}

\bibliography{main}

\vspace{.5em}

\end{document}